\documentclass{amsart}

\usepackage{amsmath, amssymb, amsthm, amsfonts,bm}

\input xy
\xyoption{all}

\usepackage{hyperref}

\swapnumbers
\numberwithin{equation}{section}

\theoremstyle{plain}
\newtheorem{theorem}[subsection]{Theorem}
\newtheorem{lemma}[subsection]{Lemma}
\newtheorem{prop}[subsection]{Proposition}
\newtheorem{cor}[subsection]{Corollary}

\theoremstyle{definition}
\newtheorem{defn}[subsection]{Definition}

\newtheorem{remark}[subsection]{Remark}
\newtheorem{exam}[subsection]{Example}

\def\CC{\mathbb{C}}

\def\FF{\mathbb{F}}

\def\NN{\mathbb{N}}

\def\RR{\mathbb{R}}

\def\ZZ{\mathbb{Z}}

\def\bR{\mathbf{R}}

\newcommand\frg{\mathfrak{g}}
\newcommand\frh{\mathfrak{h}}

\newcommand{\ch}{\textup{char}}

\newcommand\id{\textup{id}}

\newcommand\Lie{\textup{Lie}\ }

\newcommand{\Nm}{\textup{Nm}}

\newcommand{\opp}{\textup{opp}}

\newcommand\Spec{\textup{Spec}\ }

\newcommand\Sym{\textup{Sym}}

\newcommand{\Trd}{\textup{Trd}}

\newcommand{\val}{\textup{val}}

\newcommand\Aut{\textup{Aut}}
\newcommand\Hom{\textup{Hom}}
\newcommand\End{\textup{End}}

\newcommand\GL{\textup{\textbf{GL}}}

\newcommand{\ad}{\textup{ad}}
\newcommand{\Ad}{\textup{Ad}}

\newcommand{\isom}{\stackrel{\sim}{\to}}

\renewcommand{\j}[1]{\langle{#1}\rangle}
\newcommand{\jj}[1]{\langle\langle{#1}\rangle\rangle}

\newcommand\quash[1]{}
\newcommand\un{\underline}

\newcommand{\ov}{\overline}

\newcommand{\lr}[1]{(\!(#1)\!)}
\newcommand\sss{\subsubsection}

\newcommand\ot{\otimes}

\renewcommand\a\alpha
\renewcommand\b\beta
\newcommand\g\gamma
\renewcommand\d\delta
\newcommand\D\Delta
\newcommand{\e}{\epsilon}
\renewcommand{\th}{\theta}
\newcommand{\ph}{\varphi}

\newcommand{\s}{\sigma}
\renewcommand{\t}{\tau}

\newcommand{\y}{\eta}
\newcommand{\z}{\zeta}
\newcommand{\ep}{\epsilon}

\newcommand\dm{\diamondsuit}

\newcommand{\bb}{\mathbf{b}}
\newcommand{\bth}{\bm{\theta}}
\renewcommand{\c}{\circ}
\newcommand\cc{\cdot,\cdot}
\newcommand{\bAut}{\textup{\textbf{Aut}}}

\title{Semilinear automorphisms of classical groups and quivers}

\dedicatory{To Professor Lo Yang's 80th anniversary, with admiration}
\thanks{Z.Y. is partially supported by the Packard Foundation.}
\author{Jinwei Yang}
\address{(J.Yang) University of Alberta, 632 Central Academic Building, Edmonton, AB, Canada T6G 2G1.}
\email{jinwei2@ualberta.ca}
\author{Zhiwei Yun}
\address{(Z.Yun) Massachusetts Institute of Technology, 77 Massachusetts Ave, Cambridge, MA 02139}\email{zyun@mit.edu}
\date{}

\begin{document}

\begin{abstract} For a classical group $G$ over a field $F$ together with a finite-order automorphism $\th$ that acts compatibly on $F$, we describe the fixed point subgroup of $\th$ on $G$ and the eigenspaces of $\th$ on the Lie algebra $\frg$ in terms of cyclic quivers with involution. More precise classification is given when $\frg$ is a loop Lie algebra, i.e., when $F=\CC\lr{t}$.
\end{abstract}

\maketitle
\tableofcontents
\section{Introduction}

\subsection{Cyclically graded Lie algebras} Let $\frg$ be a Lie algebra of a connected simple algebraic group $G$ over a  field $k$. A {\em cyclic grading} on $\frg$ is a decomposition
\begin{equation*}
\frg = \bigoplus_{i \in \ZZ/m\ZZ}\frg_i,
\end{equation*}
where $m \in \NN$ and $[\frg_i, \frg_j] \subset \frg_{i+j}$ for all $i, j\in\ZZ/m\ZZ$. The summand $\frg_0$ is a Lie subalgebra; let $G_0$ denote the corresponding connected subgroup of $G$.  When $m$ is prime to $\ch(k)$ and $k$ contains all $m$-th roots of unity, such a cyclic grading corresponds to an automorphism $\th$ of $\frg$ of order divisible by $m$, under which $\frg_{i}$ is the eigenspace of $\th$ with eigenvalue $\z^{i}$ (where $\z$ is a fixed primitive $m$-th root of unity).

The invariant theory of the action of $G_{0}$ on $\frg_{i}$ has been much studied by Vinberg and his school. The $G_{0}$ action on $\frg_{i}$ share many nice properties of the adjoint action of $G$ on $\frg$. In \cite{RLYG}, the authors single out {\em stable gradings} (when $\frg_{i}$ has stable vectors under $G_{0}$)  and connect them with regular elliptic elements in the Weyl group of $G$. When $\frg$ is a classical Lie algebra, the subgroup $G_{0}$ as well as its action on $\frg_{i}$ can be described in terms of cyclic quivers with involution, see \cite[\S6-8]{Y}.

From the Lie-theoretic perspective it is natural to consider cyclic gradings on Kac-Moody algebras, starting from the loop Lie algebras. In the case of a loop Lie algebra $\frg\ot k\lr{t}$, it is interesting to consider not only those cyclic gradings coming from $k\lr{t}$-linear automorphisms of $\frg$, but also those coming from $k\lr{t}$-semilinear automorphisms. For example, let $\z$ be a root of unity in $k$, we may consider a finite order automorphism $\th$ of $\frg\ot k\lr{t}$ such that $\th(X\ot a(t))=\th(X)\ot a(\z t)$ for all $a(t)\in k\lr{t}$ and $X\in \frg$. In this paper we generalize the quiver description in \cite{Y} for cyclic gradings on classical Lie algebras to a setting that includes finite-order semilinear automorphisms of loop Lie algebras of classical type.

\subsection{Convention}\label{ss:con} For a field $k$ and $n\in \NN$, $\mu_{n}(k)$ denotes the group of $n$-th roots of unity in $k^{\times}$.

Let $A$ be an associative ring, and  $M, M'$ two left $A$-modules. Let $\z$  be an automorphism of $A$. A map $f:M\to M'$ is called {\em $(A,\z)$-semilinear} if
\begin{equation*}
f(av)=\z(a)f(v), \quad \forall a\in A, v\in M.
\end{equation*}

For a left $A$-module $M$, let $\End_{A}(M)$ denote the set of $A$-linear endomorphisms of $M$ and $\Aut_{A}(M)$ denote the group of $A$-linear automorphisms of $M$.

If $A$ contains a field $k$, and $M$ is a left $A$-module $M$ of finite dimension over $k$, let $\GL_{A/k}(M)$ be the algebraic group over $k$ whose $R$-points (for any commutative $k$-algebra $R$) are $R\ot_{k}A$-linear automorphisms of $R\ot_{k}M$. When $A=k$ we write $\GL_{k}(M)$ for $\GL_{k/k}(M)$, which is the usual general linear group. If $A$ is commutative, then $\GL_{A/k}(M)=\bR_{A/k}\GL_{A}(M)$ is the Weil restriction of the general linear group $\GL_{A}(M)$ from $A$ to $k$.

If, moreover, the $A$-module $M$ carries a $k$-bilinear pairing $\j{\cc}: M\times M\to B$ valued in some $k$-vector space $B$, we denote by $\bAut_{A/k}(M,\j{\cc})$ the algebraic subgroup  of $\GL_{A/k}(M)$ preserving the pairing.

\subsection{The setup and the main result}\label{ss:setup}
Throughout the paper, let $k$ be a field. Let $F$ be a finite separable $k$-algebra together with an automorphism $\z\in \Aut(F)$ of order $n\in\NN$ such that $k=F^{\z}$. We allow $F$ to be a product of fields.

Let $V$ be a finite type $F$-module. Let $\th$ be an $(F,\z)$-semilinear automorphism of $V$. Let $m$ be a multiple of $n$ such that $m/n$ is invertible in $k$. Assume
\begin{equation*}
\th^{m}=\b\cdot\id _{V}, \quad\textup{ for some } \b\in k^{\times}.
\end{equation*}


Then $\th$ acts on the Weil restriction $\GL_{F/k}(V)$ and on the Lie algebra $\End_{F}(V)$ by conjugation. As a warm-up, in \S\ref{s:linear} we describe the fixed point subgroup of $\th$ on $\GL_{F/k}(V)$ and the $\th$-eigenspaces  on $\End_{F}(V)$ in terms of cyclic quivers decorated by division algebras.  The more complicated case where  $\GL_{F/k}(V)$ is replaced with a classical group  $G$ defined using $V$ and a symmetric bilinear form, a symplectic form or a Hermitian form on it is considered in \S\ref{s:pol}.

Our main result is Theorem \ref{th:main}, which gives a complete description of the fixed subgroup $H$ of $\th$ on $G$ and eigenspaces $\frg(\xi)$ of $\th$ on $\frg=\Lie G$ in terms of cyclic quivers with involution decorated by division algebras and pairings. The strategy of the proof is to realize $V$ as a module over a certain semisimple (non-commutative) algebra $A_{\b}$, and to extract linear-algebraic data from the multiplicity spaces of  simple $A_{\b}$-modules in $V$.

In \S\ref{s:loop} we specialize to the case of classical loop Lie algebras, and make the description in Theorem \ref{th:main} more precise. The result in this case can be summarized in the following rough form: when $G$ comes from a polarization on $V$ and $\Nm_{F/k}(\xi)$ is a primitive $m/n$-th root of unity, we can associate to the situation a  cyclic quiver $Q_{\xi}$ with $m/n$ or $m/2n$  vertices (i.e., there is a vector space $M_{i}$ on each vertex $i$ of $Q_{\xi}$ over $k$ or a quadratic extension of $k$). The quiver $Q_{\xi}$ is equipped with an involution $(-)^{\dm}$, and the vector spaces on $i$ and $i^{\dm}$ are dual to each other. For $i=i^{\dm}$, $M_{i}$ is equipped with a symmetric bilinear, skew-symmetric bilinear  or Hermitian form. Then $H=G^{\Ad(\th)}$ is the automorphism group of the $(M_{i})_{i\in I}$ preserving the pairings and forms; $\frg(\xi)$ is the space of representations of the quiver $Q_{\xi}$ in the vector spaces $(M_{i})$ satisfying a certain self-adjointness conditions with respect to the pairings.

\subsection{Examples of the setup}
\begin{enumerate}
\item $n=1$ so $k=F$. In this case $V$ is a finite-dimensional $k$-vector space with a $k$-linear operator $\th$ such that $\th^{m}$  is a scalar.
\item $m=n$. In this case $\th$ gives a descent datum of $V$ to a $k$-vector space $V'$, i.e., $V=V'\ot_{k} F$.
\item $k=\RR$, $F=\CC$ and $n=2$. In this case $V$ is a complex vector space with a complex {\em anti-linear} automorphism $\th$ of finite order.
\item $k$ is a discrete valuation field and $F$ is a tamely ramified Galois extension of $k$ of degree $n$. This includes the case $k=\CC\lr{t^{n}}$ and $F=\CC\lr{t}$ with the action $\z( a(t))=a(\z_{n}t)$ for some primitive $n$-th root of unity $\z_{n}$, which arises from the loop Lie algebra setting discussed in the beginning.
\item Let $k$ be a field containing a finite field $\FF_{q}$, and $F=k\ot_{\FF_{q}}\FF_{q^{n}}$ with the action of $\z$ by $q$-Frobenius on the $\FF_{q^{n}}$-factor.  An $F$-vector space $V$ with an $(F,\z)$-semilinear automorphism $\th$ appears in the study of the generic fiber of Shtukas by Drinfeld \cite[\S2]{D} (which Drinfeld calls ``F-spaces'').
\end{enumerate}

\section{Linear case}\label{s:linear}

\subsection{The problem}\label{ss:linear setup}

%
We are in the setup of \S\ref{ss:setup}. Let $G=\GL_{F}(V)$, the general linear group over $F$.  Let $\frg=\End_{F}(V)$ be the Lie algebra of $G$. Let $\Ad(\th)$ (resp. $\ad(\th)$) denote the conjugation action of $\th$ on $G$ (resp. $\frg$): $\Ad(\th)(g)=\th g\th^{-1}$ for $g\in G$ (resp. $\ad(\th)(\ph)=\th\ph\th^{-1}$ for $\ph\in \frg$). Then $\Ad(\th)$ is an automorphism of the Weil restriction $\bR_{F/k}G=\GL_{F/k}(V)$, and $\ad(\th)$ is an $(F,\z)$-semilinear automorphism of $\frg$.  Our goal is to understand the following in terms of quivers:
\begin{enumerate}
\item The fixed point group $H:=(\bR_{F/k}G)^{\Ad(\th)}$ as an algebraic group over $k$. Note that $H(k)=\{g\in \Aut_{F}(V)|g\th=\th g\}$.
\item For $\xi\in F^{\times}$,  the $H$-module
\begin{equation*}
\frg(\xi):=\{\ph\in \End_{F}(V)| \th\ph\th^{-1}=\xi \ph\}.
\end{equation*}
with the action of $h\in H$ by $h: \ph\mapsto h\ph h^{-1}$.
\end{enumerate}

Since $\ad(\th)$ is not $F$-linear, $\frg(\xi)$ is not an eigenspace of $\ad(\th)$ in the traditional sense. In particular, for different $\xi$, the subspaces $\frg(\xi)$ are not necessarily linearly independent.

\begin{defn}\label{def:A}
\begin{enumerate}
\item Let $F\j{\bth}$ be the non-commutative polynomial ring over $F$ in one variable $\bth$ with the relation $\bth a=\z(a)\bth$ for all $a\in F$.
\item Let $A_{\b}$ be the quotient of $F\j{\bth}$ by the ideal generated by the central element $\bth^{m}-\b$.
\end{enumerate}
\end{defn}

By construction, $A_{\b}$ is an associative $F$-algebra. An $A_{\b}$-module  is an $F$-vector space $U$ together with an $(F,\z)$-semilinear automorphism $T: U\to U$ satisfying $T^{m}=\b\cdot\id_{U}$. In particular, $V$ is an $A_{\b}$-module with  $\bth$ acting by $\th$.


\subsection{Twisting by $\xi$}\label{ss:tw xi}
Let
\begin{equation*}
\Xi_{m/n}=\{\xi\in F^{\times} | \Nm_{F/k}(\xi)\in \mu_{m/n}(k)\}.
\end{equation*}
For $\xi\in \Xi_{m/n}$, let $\mu_{\xi}$ be the  $F$-linear automorphism of $A_{\b}$ sending  $\bth$ to $\xi\bth$. This defines an action of $\Xi_{m/n}$ on $A_{\b}$. For an $A_{\b}$-module $V$, let $V^{\xi}$ be the same $F$-vector space $V$  with the action of $A_{\b}$ twisted by $\mu_{\xi}$; i.e., the new action of $\bth$ on $v\in V^{\xi}$ is $\bth\cdot  v=\xi\th(v)$.



\subsection{Reformulation of the problem}\label{ss:ref GL} We may rewrite $H$ and $\frg(\xi)$ in terms of the $A_{\b}$-module structure on $V$:
\begin{enumerate}
\item $H=\GL_{A_{\b}/k}(V)$ as an algebraic group over $k$ (see \S\ref{ss:con} for convention);
\item For $\xi\in \Xi_{m/n}$, we have $\frg(\xi)=\Hom_{A_{\b}}(V^{\xi}, V)$. Note that if $\xi\notin \Xi_{m/n}$, then $\frg(\xi)=0$.
\end{enumerate}

\subsection{Classification of $A_{\b}$-modules}\label{ss:L}
Let $L_{\b}=k[\bth^{n}]\subset A_{\b}$; this is the center of $A_{\b}$. Let $b=\bth^{n}\in L_{\b}$, then $L_{\b}\cong k[\bb]/(\bb^{m/n}-\b)$ (the image of $\bb$ in $L_{\b}$ is $b$) is a separable $k$-algebra (since $m/n$ is prime to $\ch(k)$ by assumption). Let
\begin{equation*}
L_{\b}=\prod_{i\in I}L_{i}
\end{equation*}
be the decomposition of $L_{\b}$ into a product of fields, with the index set $I$ in natural bijection with the underlying set of $\Spec L_{\b}$. Let $b_{i}\in L_{i}$ be the image of $b$. Then
\begin{equation*}
A_{\b}=\prod_{i\in I}A_{i}, \quad \textup{ with }A_{i}=(L_{i}\ot_{k} F\j{\bth})/(\bth^{n}-b_{i}).
\end{equation*}

\begin{lemma}
The algebra $A_{i}$ is a central simple algebra over $L_{i}$.
\end{lemma}
\begin{proof} The presentation of $A_{i}$ is the standard one for a cyclic algebra of degree $n^{2}$ over $L_{i}$. In particular  $A_{i}$ is a central simple algebra over $L_{i}$.
\end{proof}

By the above lemma, for each $i\in I$, there is up to isomorphism a unique simple $A_{i}$-module. {\em We fix a simple $A_{i}$-module $S_{i}$ for each $i\in I$}. Let $D_{i}=\End_{A_{i}}(S_{i})^{\opp}$. Then $D_{i}$ is a central division algebra over $L_{i}$. Let $n_{i}=\dim_{D^{\opp}_{i}}(S_{i})$, then $A_{i}=\End_{D_{i}^{\opp}}(S_{i})\cong M_{n_{i}}(D_{i})$, and $\dim_{L_{i}}(D_{i})=(n/n_{i})^{2}$.
We view $S_{i}$ as a right $D_{i}$-module with the right $D_{i}$-action given by the left $D_{i}^{\opp}=\End_{A_{i}}(S_{i})$-action on $S_{i}$.


\begin{cor} The algebra $A_{\b}$ is a semisimple $k$-algebra with the set of simple modules up to isomorphism given by $\{S_{i}\}_{i\in I}$. Any $A_{\b}$-module $V$ is canonically isomorphic to a direct sum
\begin{equation}\label{V decomp}
V\cong\bigoplus_{i\in I} S_{i}\ot_{D_{i}}M_{i}
\end{equation}
where $M_{i}=\Hom_{A_{i}}(S_{i}, V)$ viewed as a left $D_{i}$-module using the right $D_{i}$-action on $S_{i}$.
\end{cor}

\subsection{The group $H$} Now we are ready to describe the group $H$ using the canonical decomposition \eqref{V decomp} for the $A_{\b}$-module $V$. We have an isomorphism of algebraic groups over $k$
\begin{equation}\label{GL H}
H=\GL_{A_{\b}/k}(V)\cong\prod_{i\in I}\GL_{D_{i}/k}(M_{i}).
\end{equation}
Under the above isomorphism, if $g\in H$ corresponds to $(g_{i})_{i\in I}$ on the right side, then
\begin{equation}\label{gi}
g(u\ot x)=u\ot g_{i}(x), \quad\forall i\in I, u\in S_{i}, x\in M_{i}.
\end{equation}

\subsection{The quiver $Q_{\xi}$}\label{ss:Q}
The action of $\xi\in \Xi_{m/n}$ on $A_{\b}$ induces an action on its center $L_{\b}$ by $\mu_{\xi}: b\mapsto \Nm_{F/k}(\xi)\cdot b$, hence a permutation on $I=\Spec L_{\b}$. We denote this permutation by $i\mapsto \ov\xi(i)$. Let $Q_{\xi}$ be the directed graph with vertex set $I$ and an arrow  $i\to \ov\xi(i)$ for each $i\in I$. Let $E$ be the set of arrows of $Q_{\xi}$. Each vertex $i\in I$ is decorated by the division algebra $D_{i}$.

In general, $Q_{\xi}$ is a disjoint union of cycles of not necessarily the same size.   In the special case where $k$ contains all $(m/n)$-th roots of unity, $L_{\b}$ is Galois over $k$, and $Q_{\xi}$ is a disjoint union of cycles of equal size.

\subsection{The $H$-module $\frg(\xi)$}\label{ss:GL g xi} Let $e:i\to \ov\xi(i)$ be an arrow in $Q_{\xi}$.
The automorphism $\mu_{\xi}$ of $A_{\b}$ restricts to an isomorphism $A_{i}\isom A_{\ov\xi(i)}$, hence a non-canonical isomorphism $\y_{e}: (S_{i})^{\xi}\cong S_{\ov\xi(i)}$. Once we fix a choice of $\y_{e}$, we get an isomorphism $\y^{\flat}_{e}: D_{i}\cong D_{\ov\xi(i)}$ by applying $\End_{A_{\b}}(-)^{\opp}$ to the source and target of $\y_{e}$. Note that even when $e$ is a self loop at $i=\ov\xi(i)$, the automorphism $\y^{\flat}_{e}$ of $D_{i}$ and even its restriction to the center $L_{i}$ may not be the identity.

We have a decomposition of $V^{\xi}$ as an $A$-module using the maps $\y_{e}$
\begin{equation*}
V^{\xi}=\bigoplus_{i\in I}(S_{i})^{\xi}\ot_{D_{i}} M_{i}\cong \bigoplus_{i\in I}S_{\ov\xi(i)}\ot_{D_{i}} M_{i}.
\end{equation*}
Here the action of $D_{i}$ on $S_{\ov\xi(i)}$ is via the isomorphism $\y^{\flat}_{e}$ for the arrow $e:i\to \ov\xi(i)$.
Hence
\begin{equation}\label{GL g xi}
\frg(\xi)=\Hom_{A_{\b}}(V^{\xi},V)=\bigoplus_{e: i\to \ov\xi(i)}\Hom_{D_{i}}(M_{i}, M_{\ov\xi(i)})
\end{equation}
where the sum runs over all arrows $e$ of $Q_{\xi}$.  Here $M_{\ov\xi(i)}$ is viewed as a $D_{i}$-module via the isomorphism $\y^{\flat}_{e}$.

Under the isomorphism \eqref{GL g xi}, if $\ph\in \Hom_{A_{\b}}(V^{\xi},V)$ corresponds to $(\ph_{e})_{e\in E}$ on the right side, then
\begin{equation}\label{phe}
\ph(u\ot x)=\y_{e}(u)\ot \ph_{e}(x), \quad\forall e:i\to \ov\xi(i), u\in S_{i}, x\in M_{i}.
\end{equation}
To summarize, $\frg(\xi)$ is the space of representations of the  quiver $Q_{\xi}$ (decorated by division algebras $D_{i}$) with a fixed dimension vector $\dim_{D_{i}}(M_{i})$ at vertex $i$.


\begin{exam} Consider the case  $F=\CC\lr{t}$ and $\z$ acts on $F$ by change of variables $t\mapsto \z_{n}t$ for some primitive $n$-th root of unity. Then $k=\CC\lr{\t}$ where $\t=t^{n}$. Without loss of generality we may assume $\b=t^{nr}=\t^{r}$ for some $r\in\ZZ$. Then $L=\CC\lr{\bb}/(\bb^{m/n}-\t^{r})$. Let $\ell=\gcd(m/n, r)$. Then $I$ can be identified with $\mu_{\ell}$, with $L_{\ep}\cong k[\bb]/(\bb^{\frac{m}{n\ell}}-\ep \t^{\frac{r}{\ell}})$ for $\ep\in \mu_{\ell}$.  We have $D_{\ep}=L_{\ep}$ since there are no nontrivial division algebras over $L_{\ep}$.

Let $\xi$ be a primitive $m$-th root of unity in $\CC$, so $\xi\in \Xi_{m/n}$. The action of $\ov\xi$ on $I=\mu_{\ell}$ is via multiplication by $\xi^{m/\ell}\in \mu_{\ell}$. In particular,   $Q_{\xi}$ is a single cycle of length $\ell$, with the vertices decorated by $L_{\ep}\cong \CC\lr{\t^{\frac{nr}{m}}}$. In this case,  we may rename the $M_{i}$ ($i\in I=\mu_{\ell}$) by $M_{0}, M_{1},\cdots, M_{\ell-1}$ so that $\frg(\xi)$ is the space of representation of the following cyclic quiver over $\CC\lr{\t^{\frac{nr}{m}}}$.
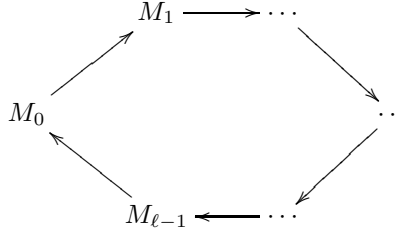
\begin{equation*}
\xymatrix{ & M_{1} \ar[r] & \cdots \ar[dr]\\
M_{0} \ar[ur]  &&& \cdots\ar[dl]\\
& M_{\ell-1}\ar[ul] &  \cdots\ar[l] & }
\end{equation*}
\end{exam}

\section{Polarized case}\label{s:pol}

In this section we extend the results of the previous section from $G=\GL_{F}(V)$ to other classical groups. We remark that even in the case $G=\GL_{F}(V)$ have not covered all finite order automorphisms of $G$ in the previous section; only inner ones are considered. The outer ones will be covered as a special case of the polarized setting in this section (see Example \ref{ex:outer A}).

We continue with the setup in \S\ref{ss:setup}. For the rest of the paper we assume $\ch(k)\ne2$.

\subsection{Involution}
Let $\s:F\to F$ be an involution that commutes with $\z$ ($\s$ maybe trivial). In particular,  $\s$ restricts to an involution on $k$.   For example, when $n$ is even, we may take $\s=\z^{n/2}$, in which case $\s|_{k}$ is trivial.

\subsection{Polarization} Let $\ep\in \{\pm1\}$. Let
\begin{equation*}
\j{\cdot, \cdot}: V\times V\to F
\end{equation*}
be a non-degenerate pairing such that
\begin{enumerate}
\item $\j{\cdot, \cdot}$ is $F$-linear in the first variable.
\item $\j{x,y}=\ep\s(\j{y,x})$. (This implies that $\j{\cdot, \cdot}$ is $(F,\s)$-semilinear in the second variable.)
\item $\j{\th x, \th y}= c\z(\j{x,y})$, for some $c\in (F^{\times})^{\s}$ such that
\begin{equation}\label{bc}
\Nm_{F/k}(c)^{m/n}=\b\s(\b).
\end{equation}
\end{enumerate}

Let $G=\bAut_{F/F^{\s}}(V, \j{\cc})\subset \GL_{F/F^{\s}}(V)$; this is an algebraic group over $F^{\s}$. Let $\frg=\Lie G$ be the Lie algebra over $F^{\s}$. When $\s$ is trivial and $\ep=1$ (resp. $\ep=-1$), $G$ is the full orthogonal group (resp. symplectic group) attached to $(V, \j{\cc})$. When $\s$ is nontrivial, we may rescale $\j{\cc}$ to reduce to the case $\e=1$, in which case $G$ is the unitary group attached to the Hermitian space $(V, \j{\cc})$. By property (3) of the pairing $\Ad(\th)$ preserves the subgroup $G$ of $\GL_{F/F^{\s}}(V)$. Similarly, $\ad(\th)$ acts on the Lie algebra $\frg$.

\begin{exam}\label{ex:outer A} Take $F=k\times k$, and let $\z=\s$ be the swapping of  two factors. In this case, we write $V=V_{0}\oplus V_{1}$ for some $k$-vector spaces $V_{0},V_{1}$ using the idempotents in $F$. The pairing $\j{\cc}$ identifies $V_{1}$ as the $k$-linear dual $V_{0}^{*}$ of $V_{0}$. The automorphism $\th$ of $V$ sends $V_{0}$ to $V_{1}=V_{0}^{*}$ and $V_{1}=V_{0}^{*}$ to $V_{0}$. We have $G=\GL_{k}(V_{0})$, and $\Ad(\th)$ is an {\em outer} automorphism of $G$.
\end{exam}

\subsection{The problem}  In the situation above, we try to understand the following in terms of quivers ``with polarizations'':
\begin{enumerate}
\item $H:=(\bR_{F^{\s}/k^{\s}}G)^{\Ad(\th)}$ as an algebraic group over $k^{\s}$. Note that $H(k^{\s})=\{g\in \Aut_{F}(V)|\j{gx,gy}=\j{x,y},\forall x,y\in V;  g\th=\th g\}$.
\item Let $\xi\in \Xi_{m/n}\cap F^{\s}$. Consider the $k^{\s}$-vector space with $H$-action
\begin{equation*}
\frg(\xi):=\{\ph\in \End_{F}(V)| \th\ph\th^{-1}=\xi \ph, \j{\ph x, y}+\j{x,\ph y}=0, \forall x,y\in V\}.
\end{equation*}
If $\xi\notin \Xi_{m/n}\cap F^{\s}$, then the similarly defined $\frg(\xi)$ is zero.
\end{enumerate}

As in \S\ref{ss:ref GL} we may describe $H$ and $\frg(\xi)$ using the $A_{\b}$-module structure on $V$:
\begin{enumerate}
\item $H=\bAut_{A_{\b}/k^{\s}}(V,\j{\cc})$.
\item For $\xi\in \Xi_{m/n}\cap F^{\s}$,
\begin{equation*}
\frg(\xi)=\{\ph\in \Hom_{A_{\b}}(V^{\xi}, V)| \j{\ph x, y}+\j{x,\ph y}=0, \forall x,y\in V\}.
\end{equation*}
\end{enumerate}

Let $n'=[F^{\s}:k^{\s}]$. Note that $n'$ is either $n$ or $n/2$ (the latter happens if and only if $\s=\z^{n/2}$). When $n'=n$, $\th^{n}$ is an $F$-linear automorphism of $V$ satisfying
\begin{equation*}
\j{\th^{n}x,\th^{n}y}=\Nm_{F/k}(c)\j{x,y}, \quad\forall x,y\in V.
\end{equation*}
If $n'=n/2$, then $\th^{n'}$ is an $(F,\s)$-semilinear automorphism of $V$ satisfying
\begin{equation*}
\j{\th^{n'}x,\th^{n'}y}=\Nm_{F^{\s}/k}(c)\s\j{x,y}, \quad\forall x,y\in V.
\end{equation*}
In any case, $\Ad(\th^{n'})$ gives an automorphism of $G$ (over $F^{\s}$), and $\ad(\th^{n'})$ gives an automorphism of $\frg$.

The following proposition describes the pair $(H,\frg(\xi))$ after base change to $F^{\s}$ in terms of the $F^{\s}$-linear action of $\th^{n'}$ on $G$ and $\frg$.

\begin{prop}\label{p:base change} We have canonical isomorphisms
\begin{eqnarray*}
&&H_{F^{\s}}\cong G^{\Ad(\th^{n'})},\\
&&\frg(\xi)\ot_{k^{\s}}F^{\s}\cong \{\ph\in \frg|\ad(\th^{n'})\ph=\xi\ph\}
\end{eqnarray*}
compatible with the natural actions of the first row on the second row.
\end{prop}
\begin{proof}
We have an isomorphism $F^{\s}\ot_{k^{\s}} V\cong V\oplus V\oplus\cdots\oplus V$ ($n'$ factors) sending $x\ot v$ to $(\z^{i}(x)v)_{0\le i\le n'-1}$. Under this isomorphism, $\id_{F^{\s}}\ot \th$ acts cyclically on the $n'$ factors, so that $\id_{F^{\s}}\ot\th^{n'}$ acts on each. The pairing $\j{\cc}$ defines a pairing on the first factor of $V$, and determines the pairings on the rest by property (3) of the pairing. An element $g\in H_{F^{\s}}$ (resp. $\ph\in \frg(\xi)\ot_{k^{\s}}F^{\s}$) is uniquely determined by its action on the first factor of $V$, on which it has to commute with $\th^{n'}$.
\end{proof}

\subsection{Duality for $A_{\b}$-modules} Let $U$ be an $A_{\b}$-module which is finite-dimensional over $F$. Let $U^{*}=\Hom_{F}(U,F)$ be the $F$-linear dual of $U$. Let $U^{\dm}$ be $U^{*}$ with the action of $F$ twisted by $\s$: i.e., for $a\in F$, $u^{*}\in U^{\dm}=U^{*}$ and $u\in U$, $(a\cdot u^{*}, u)=\s(a)(u^{*},u)$, where $(u^{*},u)$ denotes the canonical pairing between $U^{*}$ and $U$. We define an $A_{\b}$-module structure on $U^{\dm}$ by requiring the action of $\bth$ to be $(F,\z)$-semilinear and  satisfy
\begin{equation*}
(\bth u^{*}, u)=c\z(u^{*},\bth^{-1} u), \quad\forall u\in U, u^{*}\in U^{\dm}.
\end{equation*}
One readily checks that $\bth^{m}u^{*}=\Nm_{F/k}(c)^{m/n}\b^{-1}u^{*}=\s(\b)u^{*}$ under the (old) $F$-action on $U^{*}$, and hence $\bth^{m}u^{*}=\b\cdot u^{*}$  under the (new) $F$-action on $U^{\dm}$.

The assignment $U\mapsto U^{\dm}$ gives a contravariant auto-equivalence on the category of finite-dimensional $A_{\b}$-modules.  On  morphisms, it sends $f: U\to W$ to the transpose $f^{\vee}:W^{\dm}=W^{*}\to U^{*}=U^{\dm}$. 

We have a canonical isomorphism of $A_{\b}$-modules $U\cong (U^{\dm})^{\dm}$  given by sending $u\in U$ to the $(F,\s)$-semilinear function $u^{*}\mapsto \s(u^{*},u)$ on $u^{*}\in U^{*}$ (which is the same as an $F$-linear function on $U^{\dm}$).

For $\xi\in \Xi_{m/n}$ (so that the twisting functor $(-)^{\xi}$ on $A_{\b}$-modules is defined as in \S\ref{ss:tw xi}), we have a canonical isomorphism
\begin{equation}\label{dmxi}
(U^{\xi})^{\dm}\cong (U^{\dm})^{\s(\xi)^{-1}}
\end{equation}
which is the identity on the underlying $F$-vector spaces.
%
%

\subsection{Involution on the quiver} We continue to use the notation $L_{\b}, I, Q_{\xi}, \ov\xi$ introduced in \S\ref{ss:L} and \S\ref{ss:Q}.


Let $\s_{c}:L_{\b}\to L_{\b}$ be the involution that is $\s$ on $k$ and $\s_{c}(b)=\Nm_{F/k}(c)b^{-1}$. The relation \eqref{bc} implies that $\s_{c}$ is a well-defined ring automorphism. It induces an involution on the set $I=\Spec L_{\b}$ which we denote by $i\mapsto i^{\dm}$. In other words $\s_{c}$ restricts to an isomorphism $L_{i}\cong L_{i^{\dm}}$. For $\xi\in\Xi_{m/n}\cap F^{\s}$, direct calculation shows that  $\s_{c}\circ \mu_{\xi} \circ \s_{c}=\mu_{\xi^{-1}}$ as automorphisms of $L_{\b}$. Therefore the involution $(-)^{\dm}$ on $I$ reverses the arrows of the quiver $Q_{\xi}$.

\subsection{Pairing between simple $A_{\b}$-modules}\label{ss:pairing}
The involution $U\mapsto U^{\dm}$ on  $A_{\b}$-modules induces an involution on the set of isomorphism classes of simple $A_{\b}$-modules. In particular, for each $i\in I$, $S_{i}^{\dm}$ is a simple $A_{\b}$-module isomorphic to $S_{i^{\dm}}$ by comparing the actions of $L_{\b}$. For each $i$, an isomorphism of $A_{\b}$-modules $\a_{i}: S_{i}^{\dm}\cong S_{i^{\dm}}$ is the same data as a perfect pairing
\begin{equation}\label{pS}
\j{\cc}_{i}: S_{i}\times S_{i^{\dm}}\to F
\end{equation}
satisfying
\begin{enumerate}
\item $\j{\cc}_{i}$ is $F$-linear in the first variable and $(F,\s)$-semilinear in the second variable.
\item $\j{\bth u,\bth v}_{i}=c\z(\j{u,v}_{i})$.
\end{enumerate}
Indeed, $\a_{i}$ determines the pairing $\j{\cc}_{i}$ characterized by $\j{u,\a_{i}(u^{*})}_{i}=(u^{*},u)$, for $u\in S_{i}, u^{*}\in S_{i}^{\dm}=S_{i}^{*}$. Conversely, any pairing $\j{\cc}_{i}$ as above induces a map $S_{i^{\dm}}\to S_{i}^{*}=S_{i}^{\dm}$ which is $F$-linear by first property and intertwines the $\bth$-action by the second. In particular, any nonzero pairing $\j{\cc}_{i}$ satisfying (1) and (2) above must be a perfect pairing. We call a pairing as in \eqref{pS} satisfying (1) and (2) above {\em admissible}.

Let $\j{\cc}_{i}$ be a perfect admissible pairing $S_{i}\times S_{i^{\dm}}\to F$. For each $d\in D_{i}$, there is a unique $\d_{i}(d)\in D_{i^{\dm}}$ such that
\begin{equation}\label{delta pairing}
\j{ud,v}_{i}=\j{u,v\d_{i}(d)}, \quad\forall u\in S_{i}, v\in S_{i^{\dm}}.
\end{equation}
Here we write the action of $D_{i}=\End_{A_{\b}}(S_{i})^{\opp}$ on $S_{i}$ as right multiplication. The assignment $d\mapsto \d_{i}(d)$ defines a $(k,\s)$-semilinear isomorphism of algebras
\begin{equation*}
\d_{i}: D_{i}\to D_{i^{\dm}}^{\opp}
\end{equation*}
which restricts to $\s_{c}: L_{i}\cong L_{i^{\dm}}$ on the centers.   Note that $\d_{i}$ depends on the choice of the admissible pairing $\j{\cc}_{i}$.

When $i=i^{\dm}$, an admissible pairing $\j{\cc}_{i}: S_{i}\times S_{i}\to F$ is called {\em Hermitian} if it satisfies $\j{v,u}_{i}=\s(\j{u,v}_{i})$ for all $u,v\in S_{i}$. It is called {\em skew-Hermitian} if it satisfies $\j{v,u}_{i}=-\s(\j{u,v}_{i})$ for all $u,v\in S_{i}$.


\begin{lemma}\label{l:pairing}
Suppose $i=i^{\dm}$. Then one of the following happens.
\begin{enumerate}
\item There exists a perfect Hermitian admissible pairing on $S_{i}$.
\item All admissible pairings on $S_{i}$ are skew-Hermitian. This can only happen when $D_{i}=L_{i}$, $\s_{c}|_{L_{i}}=\id$ and $\s\ne\id_{F}$ (in particular, $\s=\z^{n/2}$ and $b_{i}^{2}=\Nm_{F/k}(c)$).
\end{enumerate}
\end{lemma}
\begin{proof}
Start with any nonzero admissible pairing $(u,v)\mapsto \jj{u,v}$ on $S_{i}$. Then  $(u,v)\mapsto \s\jj{v,u}$ is another admissible pairing. Therefore, if $\jj{u,v}+\s\jj{v,u}$ is not identically zero, it gives a perfect Hermitian admissible pairing.

Now suppose a  perfect Hermitian admissible pairing on $S_{i}$ does not exist. This means $\jj{u,v}+\s\jj{v,u}=0$ for any $u,v\in S_{i}$ and any admissible pairing  $\jj{\cc}$ on $S_{i}$. In other words, all admissible pairings on $S_{i}$ are skew-Hermitian.  Pick any perfect skew-Hermitian admissible pairing  $\j{\cc}_{i}$ on $S_{i}$. Let $\d_{i}: D_{i}\isom D_{i}^{\opp}$ be the corresponding isomorphism characterized by \eqref{delta pairing}. For $d\in D_{i}$, the pairing $\j{u,v}_{d}:=\j{ud,v}_{i}$ is also admissible, hence also skew-Hermitian. Then we have $\s\j{ud,v}_{i}=\s\j{u,v}_{d}=-\j{v,u}_{d}=-\j{vd,u}_{i}=-\j{v,u\d_{i}(d)}_{i}=\s\j{u\d_{i}(d), v}_{i}$ for all $u,v\in S_{i}$, hence $\d_{i}(d)=d$ for all $d\in D_{i}$. In this case, $D_{i}$ must be commutative since $\d_{i}$ is an anti-automorphism of $D_{i}$. Hence $D_{i}=L_{i}$. Since $\d_{i}$ restricts to $\s_{c}$ on $L_{i}$, we must have  $\s_{c}|_{L_{i}}=\id$.

It remains to show that $\s\ne\id_{F}$ when the above situation happens. Suppose in contrary that $\s=\id_{F}$, then $\j{\cc}_{i}$ is skew-symmetric, hence $\j{u,u}_{i}=0$ for any $u\in S_{i}$. Moreover, $\j{(a\ot \ell)u,u}_{i}=a\ell\j{u,u}=0$ for any $a\in F, \ell\in L_{i}$. Since $D_{i}=L_{i}$, we have $A_{i}\cong M_{n}(L_{i})$ and $F\ot_{k}L_{i}$ is maximal abelian subalgebra in $A_{i}$. Hence $S_{i}$ is a rank one free $F\ot_{k}L_{i}$-module. If we choose $u\in S_{i}$ generating $S_{i}$ as an  $F\ot_{k}L_{i}$-module,  then $\j{S_{i},u}_{i}=0$, contradicting the fact that $\j{\cc}_{i}$ is a perfect pairing. This finishes the argument.
\end{proof}

\subsection{Choice of admissible pairings} For the rest of the section, for each $i=i^{\dm}$ we fix a perfect Hermitian admissible pairing $\j{\cc}_{i}$ on $S_{i}$ if there exists one; otherwise we fix a perfect skew-Hermitian admissible pairing $\j{\cc}_{i}$ on $S_{i}$. Moreover, for $i\ne i^{\dm}$, we choose perfect admissible pairings  $\j{\cc}_{i}$ and $\j{\cc}_{i^{\dm}}$ such that
\begin{equation}\label{Herm}
\j{v,u}_{i^{\dm}}=\s(\j{u,v}_{i}), \quad\forall u\in S_{i}, v\in S_{i^{\dm}}.
\end{equation}
By our choice, for each $i\in  I$, there is a sign $\e_{i}\in \{\pm1\}$ such that
\begin{equation}\label{pairing comp}
\j{v,u}_{i^{\dm}}=\e_{i}\s(\j{u,v}_{i}), \quad\forall u\in S_{i}, v\in S_{i^{\dm}}.
\end{equation}
Moreover, the case $\e_{i}=-1$ can happen only in the situation (2) of Lemma \ref{l:pairing}.

Define $\d_{i}: D_{i}\isom D_{i^{\dm}}^{\opp}$ using the chosen $\j{\cc}_{i}$ as in \S\ref{ss:pairing}. The property \eqref{pairing comp} implies
\begin{equation*}
\d_{i^{\dm}}\c\d_{i}=\id_{D_{i}}.
\end{equation*}
In particular, for $i=i^{\dm}$, $\d_{i}$ is an anti-involution on $D_{i}$.


%

\begin{lemma}\label{l:pre M}
There is a unique pairing
\begin{equation*}
\{\cc\}'_{i}: M_{i}\times M_{i^{\dm}}\to D_{i}
\end{equation*}
characterized by the following property
\begin{equation}\label{xy}
\j{u\ot x, v\ot y}=\j{u\{x,y\}'_{i},v}_{i}, \quad\forall u\in S_{i}, v\in S_{i^{\dm}}, x\in M_{i}, y\in M_{i^{\dm}}.
\end{equation}
Moreover, the pairing $\{\cc\}'_{i}$ satisfies the following identities for all $x\in M_{i}, y\in M_{i^{\dm}}$ and $d\in D_{i}$
\begin{eqnarray}
\label{dx}  \{dx, y\}'_{i}=d\{x,y\}'_{i}; \\
\label{dy}   \{x, \d_{i}(d)y\}'_{i}=\{x,y\}'_{i}d; \\
\label{sw xy} \{y,x\}'_{i^{\dm}}=\e\e_{i}\d_{i}(\{x,y\}'_{i}).
\end{eqnarray}
(indeed \eqref{dy} follows from \eqref{dx} and \eqref{sw xy}).
\end{lemma}
\begin{proof}
For fixed $u\in S_{i}, x\in M_{i}$ and $y\in M_{i^{\dm}}$, the assignment $v\mapsto \j{u\ot x, v\ot y}$ is an $(F,\s)$-semilinear function $S_{i^{\dm}}\to F$, and therefore can be written as $v\mapsto\j{u',v}_{i}$ for a unique $u'\in S_{i}$. The assignment $u\mapsto u'$ gives an $F$-linear endomorphism of $S_{i}$. We claim that $u\mapsto u'$ is  moreover $A_{\b}$-linear. Indeed, it is enough to check $\j{\bth u\ot x,v\ot y}=\j{\bth u',v}_{i}$, which follows by comparing property (3) of $\j{\cc}$ and property (2) of $\j{\cc}_{i}$. Since $u\mapsto u'$ is $A_{\b}$-linear, there is a unique $d\in D_{i}$ such that $u'=ud$ for all $u\in S_{i}$. We then define $\{x,y\}'_{i}=d\in D_{i}$. The properties \eqref{dx} and \eqref{dy} are easy to verify using \eqref{delta pairing}. The property \eqref{sw xy} is verified using property (2) of the pairing $\j{\cc}$ on $V$ and property \eqref{pairing comp} of the pairings $\j{\cc}_{i}$.
\end{proof}

\begin{defn} Let $\{\cc\}_{i}$ be the $L_{i}$-valued pairing
\begin{eqnarray*}
\{\cc\}_{i}: M_{i}\times M_{i^{\dm}}&\to& L_{i}\\
(x,y) &\mapsto & \Trd_{D_{i}/L_{i}}\{x,y\}'_{i}
\end{eqnarray*}
where $\Trd_{D_{i}/L_{i}}: D_{i}\to L_{i}$ is the reduced trace.
\end{defn}

\begin{remark}\label{r:Trd} The $D_{i}$-valued pairing $\{\cc\}'_{i}$ satisfying \eqref{dx} can be recovered from the $L_{i}$-valued pairing $\{\cc\}_{i}$. Indeed, $\{x,y\}'_{i}$ is the unique element $z\in D_{i}$ such that $\Trd_{D_{i}/L_{i}}(dz)=\{dx,y\}_{i}$ for all $d\in D_{i}$.
\end{remark}

\begin{cor}[of Lemma \ref{l:pre M}]\label{c:Lpairing} The pairing $\{\cc\}_{i}$ is $L_{i}$-linear in the first variable and $(L_{i},\s_{c})$-semilinear in the second variable, and
\begin{eqnarray*}
\{y,x\}_{i^{\dm}}=\e\e_{i}\s_{c}(\{x,y\}_{i}), \quad\forall x\in M_{i}, y\in M_{i^{\dm}}.
\end{eqnarray*}
\end{cor}

\begin{lemma}\label{l:H}
Let $g\in \Aut_{A_{\b}}(V)$ correspond to a family of automorphisms $g_{i}\in \Aut_{D_{i}}(M_{i})$ under \eqref{GL H}. Then $g$ preserves the form $\j{\cc}$ on $V$ if and only if for all $i\in I$,
\begin{equation}\label{gi pairing}
\{x,y\}_{i}=\{g_{i}(x), g_{i^{\dm}}(y)\}_{i}, \quad\forall x\in M_{i}, y\in M_{i^{\dm}}.
\end{equation}
\end{lemma}
\begin{proof}
By \eqref{gi}, for $u\in S_{i}, v\in S_{i^{\dm}}, x\in M_{i}$ and $y\in M_{i^{\dm}}$, we have
\begin{equation*}
\j{g(u\ot x), g(v\ot y) }=\j{u\ot g_{i}(x), v\ot g_{i^{\dm}}(y)}=\j{u\{g_{i}(x), g_{i^{\dm}}(y)\}'_{i}, v}_{i}.
\end{equation*}
Comparing with \eqref{xy} we get $\{x,y\}'_{i}=\{g_{i}(x), g_{i^{\dm}}(y)\}'_{i}$. By Remark \ref{r:Trd}, this is equivalent to \eqref{gi pairing}.
\end{proof}

After fixing the admissible pairings between the simple $A_{\b}$-modules, we are going to choose a family of $A_{\b}$-module isomorphisms $\y_{e}:S_{i}^{\xi}\to S_{\ov\xi(i)}$  for each arrow $e:i\to \ov\xi(i)$ of $Q_{\xi}$.

Let $e:i\to \ov\xi(i)$ be an arrow in $Q_{\xi}$ such that $e=e^{\dm}$, i.e., $\ov\xi(i)=i^{\dm}$ (the case $i=i^{\dm}$ is allowed). Note that in this case $\s_{c\xi^{-1}}: \bb\mapsto \Nm_{F/k}(c\xi^{-1})\bb^{-1}$ is an automorphism of $L_{i}$.

An isomorphism of $A_{\b}$-modules $\y_{e}: S_{i}^{\xi}\isom S_{\ov\xi(i)}=S_{i^{\dm}}$ is called {\em self-adjoint} if $\j{u, \y_{e}(v)}_{i}=\s\j{v,\y_{e}(u)}_{i}$ for all $u,v\in S_{i}=S_{i}^{\xi}$; $\y_{e}$ is called {\em skew-self-adjoint} if $\j{u, \y_{e}(v)}_{i}=-\s\j{v,\y_{e}(u)}_{i}$ for all $u,v\in S_{i}=S_{i}^{\xi}$.

\begin{lemma}\label{l:ye}
Let $e:i\to \ov\xi(i)$ be an arrow in $Q_{\xi}$ such that $e=e^{\dm}$. Then one of the following happens.
\begin{enumerate}
\item There exists a self-adjoint $A_{\b}$-linear isomorphism $\y_{e}: S_{i}^{\xi}\isom S_{\ov\xi(i)}=S_{i^{\dm}}$.
\item All elements in $\Hom_{A_{\b}}(S_{i}^{\xi}, S_{i^{\dm}})$ are skew-self-adjoint. This can only happen when $D_{i}=L_{i}$, $\s_{c\xi^{-1}}|_{L_{i}}=\id$ and $\s\ne\id_{F}$ (in particular, $\s=\z^{n/2}$ and $b_{i}^{2}=\Nm_{F/k}(c\xi^{-1})$).
\end{enumerate}

\end{lemma}
\begin{proof}
The argument is similar to that of Lemma \ref{l:pairing}. For any $A_{\b}$-linear map $\y: S_{i}^{\xi}\to S_{\ov\xi(i)}$, define $\y^{*}: S_{i}^{\xi}\to S_{\ov\xi(i)}$ by requiring $\j{u,\y(v)}_{i}=\s\j{v,\y^{*}(u)}_{i}$ for all $u,v\in S_{i}$. Then $\y^{*}$ is also  $A_{\b}$-linear, and $\y^{**}=\y$. If $\y+\y^{*}$ is nonzero, it gives a self-adjoint isomorphism.

Now suppose a self-adjoint $\y$ does not exist. This implies $\y+\y^{*}=0$ for all $\y\in \Hom_{A_{\b}}(S_{i}^{\xi}, S_{i^{\dm}})$, i.e., all $\y$ are skew-self-adjoint. Fix a skew-self-adjoint isomorphism  $\y_{e}:S_{i}^{\xi}\isom S_{i^{\dm}}$. For any $d\in D_{i}$, $u\mapsto \y_{e}(ud)$ again belongs to $\Hom_{A_{\b}}(S_{i}^{\xi}, S_{i^{\dm}})$, hence it is also skew-self-adjoint. Therefore,  for $u,v\in S_{i}$, $\j{u,\y_{e}(v)\y_{e}^{\flat}(d)}_{i}=\j{u,\y_{e}(vd)}_{i}=-\s\j{v,\y_{e}(ud)}_{i}=\j{ud,\y_{e}(v)}_{i}=\j{u,\y_{e}(v)\d_{i}(d)}_{i}$. Hence $\y_{e}^{\flat}=\d_{i}$ as maps $D_{i}\to D_{i^{\dm}}$. Since $\y_{e}^{\flat}$ is an algebra isomorphism while $\d_{i}$ is an anti-isomorphism, we conclude that $D_{i}$ is commutative hence $D_{i}=L_{i}$. Moreover, $\mu_{\xi}|_{L_{i}}=\y_{e}^{\flat}|_{L_{i}}=\d_{i}|_{L_{i}}=\s_{c}|_{L_{i}}$, which implies  $\s_{c\xi^{-1}}|_{L_{i}}=\id$. Finally, to rule out the case $\s=\id_{F}$, we use the same argument as in  Lemma \ref{l:pairing}. By skew-self-adjointness we have $\j{au\ell,\y_{e}(u)}_{i}=0$ for all $a\in F,\ell\in L_{i}$ and $u\in S_{i}$; choosing $u$ to be a generator of the rank one $F\ot_{k} L_{i}$-module $S_{i}$ we get $\j{S_{i},\y_{e}(u)}_{i}=0$ which is a contradiction.
\end{proof}

\subsection{Choice of the isomorphisms $\y_{e}$}
Next, for each arrow $e:i\to i^{\dm}$ in $Q_{\xi}$, we fix an isomorphisms of $A_{\b}$-modules
\begin{equation*}
\y_{e}: S_{i}^{\xi}\isom S_{\ov\xi(i)}
\end{equation*}
as follows. If $e\ne e^{\dm}$ (say $e:i\to \ov\xi(i)$, $e^{\dm}: \ov\xi(i)^{\dm}\to i^{\dm}$), then we choose $\y_{e}$ and $\y_{e^{\dm}}$ so that
\begin{equation*}
\s\j{v,\y_{e}(u)}_{\ov\xi(i)^{\dm}}=\j{u,\y_{e^{\dm}}(v)}_{i}, \quad\forall u\in S_{i}, v\in S_{\ov\xi(i)^{\dm}}.
\end{equation*}
If $e=e^{\dm}$, we choose $\y_{e}$ to be self-adjoint if there exists one; otherwise  we choose $\y_{e}$ to be skew-self-adjoint. By our choice, for each arrow $e$ there is a sign $\e_{e}\in \{\pm1\}$ such that
\begin{equation}\label{adj y}
\s\j{v,\y_{e}(u)}_{\ov\xi(i)^{\dm}}=\e_{e}\j{u,\y_{e^{\dm}}(v)}_{i}, \quad\forall u\in S_{i}, v\in S_{\ov\xi(i)^{\dm}}.
\end{equation}
Moreover,  $\e_{e}=-1$ can only happen in the situation (2) of Lemma \ref{l:ye}.

\begin{lemma}\label{l:g xi}
Let $\ph\in \Hom_{A_{\b}}(V^{\xi},V)$ correspond to a family of maps $\ph_{e}\in \Hom_{D_{i}}(M_{i}, M_{\ov\xi(i)})$ for each arrow $e:i\to \ov\xi(i)$ in $Q_{\xi}$, see \eqref{GL g xi}. Then $\ph\in \frg(\xi)$ if and only if for each arrow $e:i\to \ov\xi(i)$,
\begin{equation}\label{phe xy}
\{y,\ph_{e}(x)\}_{\ov\xi(i)^{\dm}}+\e\e_{e}\s_{c\xi^{-1}}(\{x,\ph_{e^{\dm}}(y)\}_{i})=0, \quad\forall x\in M_{i}, y\in M_{\ov\xi(i)^{\dm}}.
\end{equation}
\end{lemma}
\begin{proof} The map $\ph$ lies in $\frg(\xi)$ if and only if
\begin{equation}\label{pre ph}
\j{\ph(u\ot x),v\ot y}+\j{u\ot x,\ph(v\ot y)}=0, \forall i\in I, u\in S_{i}, x\in M_{i}, v\in S_{\ov\xi(i)^{\dm}}, y\in M_{\ov\xi(i)^{\dm}}.
\end{equation}
Let $e:i\to \ov\xi(i)$ so that $e^{\dm}: \ov\xi(i)^{\dm}\to i^{\dm}$. We have by \eqref{phe} and \eqref{xy}
\begin{equation}
\j{\ph(u\ot x), v\ot y}=\j{\y_{e}(u)\ot \ph_{e}(x), v\ot y}_{\ov\xi(i)}=\j{\y_{e}(u)\{\ph_{e}(x),y\}'_{\ov\xi(i)}, v}_{\ov\xi(i)}.
\end{equation}
By \eqref{pairing comp} the above is equal to $\e_{\ov\xi(i)}\j{v,\y_{e}(u)\{\ph_{e}(x),y\}'_{\ov\xi(i)}}_{\ov\xi(i)^{\dm}}$. Hence
\begin{equation}\label{phuxvy}
\j{\ph(u\ot x), v\ot y}=\e_{\ov\xi(i)}\j{v,\y_{e}(u)\{\ph_{e}(x),y\}'_{\ov\xi(i)}}_{\ov\xi(i)^{\dm}}.
\end{equation}
On the other hand, by \eqref{phe},  and \eqref{xy}
\begin{equation*}
\j{u\ot x, \ph(v\ot y)}=\j{u\ot x, \y_{e^{\dm}}(v)\ot \ph_{e^{\dm}}(y)}_{i}=\j{u\{ x,\ph_{e^{\dm}}(y)\}'_{i}, \y_{e^{\dm}}(v)}_{i}.
\end{equation*}
By \eqref{adj y} and the definition of $\y^{\flat}_{e}$, we have
\begin{equation*}
\j{u\{ x,\ph_{e^{\dm}}(y)\}'_{i}, \y_{e^{\dm}}(v)}_{i}=\e_{e}\s\j{v,\y_{e}(u\{x,\ph_{e^{\dm}}(y)\}'_{i})}_{\ov\xi(i)^{\dm}}=\e_{e}\s\j{v,\y_{e}(u)\y^{\flat}_{e}(\{x,\ph_{e^{\dm}}(y)\}'_{i})}_{\ov\xi(i)^{\dm}}.
\end{equation*}
Therefore
\begin{equation}\label{uxphvy}
\j{u\ot x, \ph(v\ot y)}=\e_{e}\j{v,\y_{e}(u)\y^{\flat}_{e}(\{x,\ph_{e^{\dm}}(y)\}'_{i})}_{\ov\xi(i)^{\dm}}.
\end{equation}
Plugging \eqref{phuxvy} and \eqref{uxphvy} into \eqref{pre ph}, we get
\begin{equation*}
\e_{\ov\xi(i)}\j{v,\y_{e}(u)\{\ph_{e}(x),y\}'_{\ov\xi(i)}}_{\ov\xi(i)^{\dm}}+\e_{e}\j{v,\y_{e}(u)\y^{\flat}_{e}(\{x,\ph_{e^{\dm}}(y)\}'_{i})}_{\ov\xi(i)^{\dm}}=0
\end{equation*}
for all $u\in S_{i}, v\in  S_{\ov\xi(i)^{\dm}}$, which is equivalent to
\begin{equation}\label{phe xy'}
\e_{\ov\xi(i)}\{\ph_{e}(x), y\}'_{\ov\xi(i)}+\e_{e}\y^{\flat}_{e}(\{x,\ph_{e^{\dm}}(y)\}'_{i})=0, \quad\forall x\in M_{i}, y\in M_{\ov\xi(i)^{\dm}}.
\end{equation}
By \eqref{sw xy} we have
\begin{equation}
\e_{\ov\xi(i)}\{\ph_{e}(x), y\}'_{\ov\xi(i)}=\e\d_{\ov\xi(i)}(\{y,\ph_{e}(x)\}'_{\ov\xi(i)^{\dm}}),
\end{equation}
hence \eqref{phe xy'} is equivalent to
\begin{equation}\label{phe xy''}
\e\d_{\ov\xi(i)}(\{y,\ph_{e}(x)\}'_{\ov\xi(i)^{\dm}})+\e_{e}\y^{\flat}_{e}(\{x,\ph_{e^{\dm}}(y)\}'_{i})=0.
\end{equation}
Taking reduced trace and using $\s_{c}\c\mu_{\xi}=\s_{c\xi^{-1}}: L_{i}\isom L_{\ov\xi(i)^{\dm}}$ we get \eqref{phe xy}, which is equivalent to \eqref{phe xy''} by Remark \ref{r:Trd}.
\end{proof}

The above lemma motivates the following definition.
\begin{defn} For an arrow $e:i\to i^{\dm}$ fixed by $(-)^{\dm}$, and a sign $\ep'\in \{\pm1\}$, we define $\frh^{\ep'}(M_{i},M_{i^{\dm}})$ to be the set of maps $\ph_{i}: M_{i}\to M_{i^{\dm}}$ such that
\begin{equation*}
\ph(dx)=\y_{e}^{\flat}(d)\ph(x),  \quad \{y,\ph(x)\}_{i}=\ep'\s_{c\xi^{-1}}(\{x,\ph(y)\}_{i}), \quad\forall d\in D_{i},x\in M_{i}, y\in M_{i^{\dm}}.
\end{equation*}
\end{defn}

\subsection{Shape of $Q_{\xi}$ with involution}\label{ss:shape}
Each connected component of $Q_{\xi}$ is a directed cycle. Let $\Pi$ be the set of connected components of $Q_{\xi}$. The involution $(-)^{\dm}$ induces an involution on $\Pi$. Let $\un \Pi$ be the set of orbits of $\Pi$ under $(-)^{\dm}$. For $\a\in \un \Pi$, let $Q^{\a}_{\xi}$ be the union of the components contained in $\a$. Note that $\#\a=1$ or $2$. We have a decomposition
\begin{equation*}
Q_{\xi}=\coprod_{\a\in \un J} Q^{\a}_{\xi}.
\end{equation*}
Corresponding to this decomposition, we have
\begin{equation*}
H=\prod_{\a\in \un\Pi}H^{\a}; \quad \frg(\xi)=\bigoplus_{\a\in \un\Pi}\frg(\xi)^{\a}
\end{equation*}
such that $H^{\a}$ acts on $\frg(\xi)^{\a}$.

The directed graph $Q^{\a}_{\xi}$ ($\a\in \un J$) with involution $(-)^{\dm}$ takes one of the follow shapes:
\begin{enumerate}
\item[(CC-$\ell$)] $Q_{\xi}^{\a}$ is a disjoint union of two direct cycles with $(-)^{\dm}$ mapping one to the other. We label the vertices by $1,\cdots, \ell,1^{\dm},\cdots, \ell^{\dm}$ as follows ($\ell\ge1$).
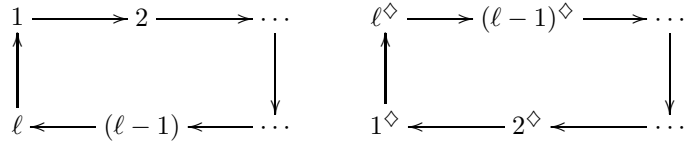
\begin{equation*}
\xymatrix{1\ar[r] & 2\ar[r] & \cdots  \ar[d]  & \ell^{\dm} \ar[r] & (\ell-1)^{\dm}\ar[r] & \cdots \ar[d] \\
\ell \ar[u]  & (\ell-1) \ar[l] & \cdots\ar[l] & 1^{\dm}\ar[u]  & 2^{\dm} \ar[l] & \cdots\ar[l]}
\end{equation*}
\item[(VV-$\ell$)] $Q^{\a}_{\xi}$ is a directed cycle with two distinct vertices and no arrow fixed by $(-)^{\dm}$. We label the vertices as follows so that $0$ and $\ell$ are fixed by $(-)^{\dm}$ ($\ell=0$ is allowed).
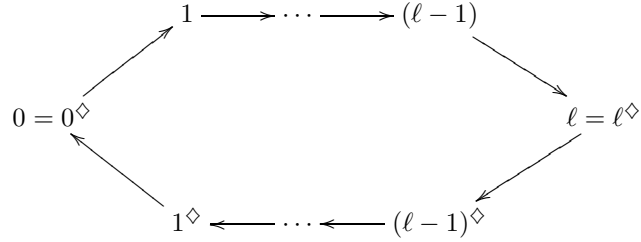
\begin{equation*}
\xymatrix{ & 1  \ar[r]  & \cdots \ar[r] & (\ell-1)\ar[dr]\\
0=0^{\dm} \ar[ur]  &&&& \ell=\ell^{\dm}\ar[dl]\\
& 1^{\dm}\ar[ul] &  \cdots  \ar[l] & (\ell-1)^{\dm}\ar[l] & }
\end{equation*}
\item[(VE-$\ell$)] $Q^{\a}_{\xi}$ is a directed cycle with exactly one vertex and one arrow fixed by $(-)^{\dm}$. We label the vertices as follows so that  $0=0^{\dm}$ and $e: \ell\to \ell^{\dm}$ is fixed by $(-)^{\dm}$ ($\ell=0$ allowed).
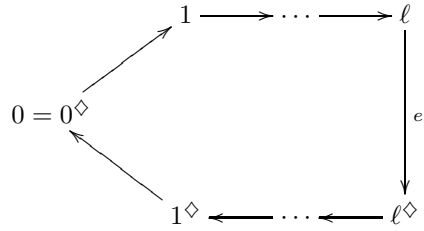
\begin{equation*}
\xymatrix{ & 1  \ar[r]  & \cdots \ar[r] & \ell\ar[dd]^{e}\\
0=0^{\dm} \ar[ur] \\
& 1^{\dm}\ar[ul] &  \cdots  \ar[l] & \ell^{\dm}\ar[l] & }
\end{equation*}
\item[(EE-$\ell$)] $Q^{\a}_{\xi}$ is a directed cycle with no vertex and exactly two arrows  fixed by $(-)^{\dm}$. We label the vertices as follows so that $e:1^{\dm}\to 1$ and $e': \ell\to \ell^{\dm}$ are fixed by $(-)^{\dm}$ ($\ell=1$ is allowed).
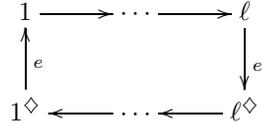
\begin{equation*}
\xymatrix{ 1  \ar[r]  & \cdots \ar[r] & \ell\ar[d]^{e'}\\
1^{\dm}\ar[u]_{e} &  \cdots  \ar[l] & \ell^{\dm}\ar[l] & }
\end{equation*}
\end{enumerate}
Our convention is such that in cases (CC-$\ell$), (VV-$\ell$), and (EE-$\ell$) the graph $Q_{\xi}^{\a}$ has $2\ell$ vertices, while in case (VE-$\ell$) it has $2\ell+1$ vertices.


\subsection{The contragredient action}\label{ss:contra} For $i\in I$ and $g\in \Aut_{D_{i}}(M_{i})$, we define $g^{*}\in \Aut_{D_{i^{\dm}}}(M_{i^{\dm}})$  so that
\begin{equation*}
\{gx,y\}_{i}=\{x,g^{*}y\}_{i},\forall x\in M_{i}, y\in M_{i^{\dm}}.
\end{equation*}
The assignment $g\mapsto g^{*,-1}$ defines an isomorphism of algebraic groups
\begin{equation*}
\GL_{D_{i}/k^{\s}}(M_{i})\cong \GL_{D_{i^{\dm}}/k^{\s}}(M_{i^{\dm}}).
\end{equation*}

For each $\a\in \un\Pi$, Lemma \ref{l:H} and Lemma \ref{l:g xi}  give a description of $H^{\a}$ and $\frg(\xi)^{\a}$ in each case classified in \S\ref{ss:shape}. We summarize our results so far in the following theorem.

\begin{theorem}\label{th:main} The isomorphism type of the directed graph $Q_{\xi}$ together with the involution $(-)^{\dm}$ on it depends only on $(k,\s|_{k}, \b, \Nm_{F/k}(c), \Nm_{F/k}(\xi))$.

For each $\a\in \un\Pi$, the pair $(H^{\a}, \frg(\xi)^{\a})$ is described as follows according to the shape of $Q_{\xi}^{\a}$.
\begin{enumerate}
\item If $Q_{\xi}^{\a}$ is of shape (CC-$\ell$), then
\begin{eqnarray*}
&&H^{\a}\cong \prod_{i=1}^{\ell} \GL_{D_{i}/k^{\s}}(M_{i}),\\
&&\frg(\xi)^{\a}\cong \oplus_{i=1}^{\ell} \Hom_{D_{i}}(M_{i}, M_{i+1}).
\end{eqnarray*}
Here $M_{\ell+1}=M_{1}$, and $M_{i+1}$ is viewed as a $D_{i}$-module by $\eta^{\flat}_{e}: D_{i}\cong D_{i+1}$ (where $e$ is the arrow $i\to i+1$). The factors $\GL_{D_{i}/k^{\s}}(M_{i})$ and $\GL_{D_{i+1}/k^{\s}}(M_{i+1})$ act on $\Hom_{D_{i}}(M_{i}, M_{i+1})$ by $(g_{i},g_{i+1})\cdot \ph=g_{i+1}\c \ph\c g_{i}^{-1}$.

\item If $Q_{\xi}^{\a}$ is of shape (VV-$\ell$), then
\begin{eqnarray*}
&&H^{\a}\cong \bAut_{D_{0}/k^{\s}}(M_{0},\{\cc\}_{0})\times \prod_{i=1}^{\ell-1} \GL_{D_{i}/k^{\s}}(M_{i})\times \bAut_{D_{\ell}/k^{\s}}(M_{\ell}, \{\cc\}_{\ell}),\\
&&\frg(\xi)^{\a}\cong \oplus_{i=0}^{\ell-1} \Hom_{D_{i}}(M_{i}, M_{i+1}).
\end{eqnarray*}
The action of $H^{\a}$ on $\frg(\xi)^{\a}$  is as explained in the case (CC-$\ell$), viewing $H^{\a}$ as a subgroup of $\prod_{i=0}^{\ell}\GL_{D_{i}/k^{\s}}(M_{i})$.

\item If $Q_{\xi}^{\a}$ is of shape (VE-$\ell$), then
\begin{eqnarray*}
&&H^{\a}\cong \bAut_{D_{0}/k^{\s}}(M_{0},\{\cc\}_{0})\times \prod_{i=1}^{\ell} \GL_{D_{i}/k^{\s}}(M_{i}),\\
&&\frg(\xi)^{\a}\cong (\oplus_{i=0}^{\ell-1} \Hom_{D_{i}}(M_{i}, M_{i+1}))\oplus \frh^{-\e\e_{e}}(M_{\ell},M_{\ell^{\dm}}).
\end{eqnarray*}
The action of $H^{\a}$ on $\Hom_{D_{i}}(M_{i}, M_{i+1})$ is as explained in the case (CC-$\ell$), viewing $\bAut_{D_{0}/k^{\s}}(M_{0},\{\cc\}_{0})$ as a subgroup of $\GL_{D_{0}/k^{\s}}(M_{0})$.
The action of $\GL_{D_{\ell}/k^{\s}}(M_{\ell})$ on $\frh^{-\e\e_{e}}(M_{\ell},M_{\ell^{\dm}})$ is induced from its natural action on $M_{\ell}$ and the contragredient action on $M_{\ell^{\dm}}$ given by $g\mapsto g^{*,-1}$ (see \S\ref{ss:contra}).

\item If $Q_{\xi}^{\a}$ is of shape (EE-$\ell$), then
\begin{eqnarray*}
&&H^{\a}\cong \prod_{i=1}^{\ell} \GL_{D_{i}/k^{\s}}(M_{i}),\\
&&\frg(\xi)^{\a}\cong \frh^{-\e\e_{e}}(M_{1^{\dm}}, M_{1})\oplus (\oplus_{i=1}^{\ell-1} \Hom_{D_{i}}(M_{i}, M_{i+1}))\oplus \frh^{-\e\e_{e'}}(M_{\ell},M_{\ell^{\dm}}).
\end{eqnarray*}
The action of $H^{\a}$ on $\Hom_{D_{i}}(M_{i}, M_{i+1})$ is as explained in the case (CC-$\ell$). The action of $\GL_{D_{\ell}/k^{\s}}(M_{1})$ on $\frh^{-\e\e_{e}}(M_{1^{\dm}}, M_{1})$ and the action of $\GL_{D_{\ell}/k^{\s}}(M_{\ell})$ on $\frh^{-\e\e_{e'}}(M_{\ell}, M_{\ell^{\dm}})$ are as explained in the (VE-$\ell$) case.
\end{enumerate}
\end{theorem}

Here is a more precise description of the factors $\bAut_{D_{i}/k^{\s}}(M_{i}, \{\cc\}_{i})$ that appear in $H$ in the above theorem. The statement follows immediately from Corollary \ref{c:Lpairing}.

\begin{prop} Let $i=i^{\dm}$ be a vertex in $Q_{\xi}$. Then
\begin{enumerate}
\item If $\s=\id_{F}$ and  $\s_{c}|_{L_{i}}=\id$,  then $\bAut_{D_{i}/k^{\s}}(M_{i}, \{\cc\}_{i})$ is an orthogonal group (resp. symplectic group) when $\e=1$ (resp. $\e=-1$).
\item If $\s\ne\id_{F}$ and $\s_{c}|_{L_{i}}=\id$ (in particular, $\s|_{k}=\id$, hence $\s=\z^{n/2}$), then $\bAut_{D_{i}/k^{\s}}(M_{i}, \{\cc\}_{i})$ is either an orthogonal or a symplectic group.
\item If $\s_{c}|_{L_{i}}\ne\id$, then $\bAut_{D_{i}/k^{\s}}(M_{i}, \{\cc\}_{i})$ is a unitary group.
\end{enumerate}
\end{prop}
\begin{proof} The map $g\mapsto g^{*}$ defined in \S\ref{ss:contra} is an anti-involution on $\End_{D_{i}}(M_{i})$. When $\s_{c}|_{L_{i}}=\id$, it is an involution of the first kind; when $\s_{c}|_{L_{i}}\ne\id$, it is an involution of the second kind. Therefore in the former case the corresponding isometry group is an orthogonal or symplectic group, while in the latter case it is a unitary group.  This proves (2) and (3).

It remains to show in the case (1), the type of $\bAut_{D_{i}/k^{\s}}(M_{i}, \{\cc\}_{i})$ is the same as that of $G$.   We already know that $\bAut_{D_{i}/k^{\s}}(M_{i}, \{\cc\}_{i})$ is either an orthogonal or a symplectic group. By Prop \ref{p:base change}, $H_{F}$ is the fixed point subgroup of $G$ (orthogonal or symplectic) under $\Ad(\th^{n})$. Hence the simple factors of $H_{F}=G^{\Ad(\th^{n})}$ are either of type $A$ or of the same type as $G$. Therefore $\bAut_{D_{i}/k^{\s}}(M_{i}, \{\cc\}_{i})$ has the same type as $G$.
\end{proof}

\section{Loop Lie algebras of classical type}\label{s:loop}
In this section we continue with the setup in \S\ref{s:pol}. We specialize to the case $k=\CC\lr{\t}$. Then $F$ is a finite separable $k$-algebra with $\Aut_{k}(F)\cong\ZZ/n\ZZ$ but we do not require $F$ to be a field.  We write $\g=\Nm_{F/k}(c)\in k^{\times}$. Let $\val_{\t}:k^{\times}\to\ZZ$ be the valuation such that $\val_{\t}(\t)=1$.

According to Theorem \ref{th:main}, the isomorphism type of $(Q_{\xi}, (-)^{\dm})$ depends only on $(k,\s|_{k}, \b,\g, \Nm_{F/k}(\xi))$.  In the following we {\em assume $\Nm_{F/k}(\xi)\in \mu_{m/n}(\CC)$ to be primitive}. We describe in more details the shape of $(Q_{\xi},(-)^{\dm})$ as well as the factors in $H$ and $\frg(\xi)$. The situation simplifies because there are no nontrivial division algebras over $L_{i}$ in this case, therefore $D_{i}=L_{i}$ for all $i\in I$.

\subsection{The case $\s|_{k}=\id$} In this case $\g^{m/n}=\b^{2}$. We distinguish two cases according to the parity of $m/n$.

\sss{$m/n$ is odd} In this case, $\val_{\t}(\b)$ is divisible by $m/n$ hence $b^{m/n}=\b$ has $m/n$ distinct solutions in $k$, i.e., $L$ splits into $m/n$-factors of $k$ (all $L_{i}=k$). The graph $Q_{\xi}$ is a single cycle of length $m/n$. Since $m/n$ is odd, it must be of type (VE).

The unique vertex $i=i^{\dm}$ corresponds to the unique $b_{i}\in k$ such that $b_{i}^{2}=\g$ and $b_{i}^{m/n}=\b$. In particular, $\s_{c}|_{L_{i}}=\id$. The factor $\bAut_{k}(M_{i}, \{\cc\}_{i})$ in $H$ is either an orthogonal group or a symplectic group over $k$. When $\s|_{F}=\id$, we have $\e_{i}=1$ by Lemma \ref{l:pairing}, hence $\bAut_{k}(M_{i}, \{\cc\}_{i})$ is an orthogonal group if $\e=1$ and a symplectic group if $\e=-1$.

The unique arrow $e:j\to j^{\dm}$ fixed by $(-)^{\dm}$ corresponds to the unique $b_{j}\in k$ such that $b_{j}^{2}=\g\Nm_{F/k}(\xi^{-1})$ and $b_{j}^{m/n}=\b$. The factor $\frh^{-\e\e_{e}}(M_{j},M_{j^{\dm}})$ in $\frg(\xi)$ is isomorphic to either $\wedge^{2}(M_{j^{\dm}})$ or $\Sym^{2}(M_{j^{\dm}})$. When $\s|_{F}=\id$, we have $\e_{e}=1$ by Lemma \ref{l:ye}, hence $\frh^{-\e\e_{e}}(M_{j},M_{j^{\dm}})$ is $\wedge^{2}(M_{j^{\dm}})$ if $\e=1$ and $\Sym^{2}(M_{j^{\dm}})$ if $\e=-1$.

\sss{$m/n$ is even} In this case we have $\b=\pm\g^{m/2n}$. Whether or not  $b^{m/n}=\b$ has a solution in $k$ depends on the parity of $\val_{\t}(\g)$.
\begin{itemize}
\item When $\val_{\t}(\g)$ is even, $L$ splits into $m/n$ factors of $L_{i}=k$. The graph $Q_{\xi}$ is a single cycle of length $m/n$. We have two subcases:
\begin{enumerate}
\item When $\b=\g^{m/2n}$, then $Q_{\xi}$ is of type (VV). Let $i,i'\in I$ be the two vertices fixed by $(-)^{\dm}$. Since $\s_{c}|_{L_{i}}=\id$ and $\s_{c}|_{L_{i'}}=\id$, the factor of $H$ corresponding to $i$ or $i'$ is either an orthogonal groups or a symplectic groups (when $\s=\id_{F}$ it is the former if $\e=1$  and the latter if $\e=-1$).

\item When $\b=-\g^{m/2n}$, then $Q_{\xi}$ is of type (EE). The factor in $\frg(\xi)$ corresponding to any  arrow $e:i\to i^{\dm}$ fixed by $(-)^{\dm}$ is isomorphic to either $\wedge^{2}(M_{i^{\dm}})$ or $\Sym^{2}(M_{i^{\dm}})$ (when $\s=\id_{F}$ it is the former if $\e=1$ and the latter if $\e=-1$).
\end{enumerate}

\item When $\val_{\t}(\g)$ is odd. In this case $L$ splits into a product of fields $L_{i}$ where each $L_{i}$ is isomorphic to the unique quadratic extension of $k$. The graph $Q_{\xi}$ is a single cycle of length $m/2n$. We have four subcases:
\begin{enumerate}
\item When $m/2n$ is odd and $\b=\g^{m/2n}$, then $Q_{\xi}$ is of type (VE). The vertex $i=i^{\dm}$ corresponds to $b_{i}^{2}=\g$, and $\s_{c}|_{L_{i}}=\id$. The factor $\bAut_{L_{i}/k}(M_{i}, \{\cc\}_{i})$ is the Weil restriction of an orthogonal group or symplectic group over $L_{i}$ (when $\s=\id_{F}$ it is the former if $\e=1$ and the latter if $\e=-1$). The edge $e:j\to j^{\dm}$ fixed by $(-)^{\dm}$ corresponds to $b_{j}^{2}=-\g\Nm_{F/k}(\xi^{-1})$, hence $\s_{c\xi^{-1}}|_{L_{j}}\ne\id$. The corresponding factor $\frh^{-\e\e_{e}}(M_{j}, M_{j^{\dm}})$ is isomorphic to the space of $L_{j}/k$-Hermitian forms on $M_{j}$.

\item When $m/2n$ is odd and $\b=-\g^{m/2n}$, then $Q_{\xi}$ is of type (VE). The vertex $i=i^{\dm}$ corresponds to $b_{i}^{2}=-\g$, and $\s_{c}|_{L_{i}}\ne\id$. The factor $\bAut_{L_{i}/k}(M_{i}, \{\cc\}_{i})$ is a unitary group over $k$.  The edge $e:j\to j^{\dm}$ fixed by $(-)^{\dm}$ corresponds to $b_{j}^{2}=\g\Nm_{F/k}(\xi^{-1})$, hence $\s_{c\xi^{-1}}|_{L_{j}}=\id$. The corresponding factor $\frh^{-\e\e_{e}}(M_{j}, M_{j^{\dm}})$ is isomorphic to either $\wedge^{2}_{L_{j^{\dm}}}(M_{j^{\dm}})$ or $\Sym^{2}_{L_{j^{\dm}}}(M_{j^{\dm}})$ (when $\s=\id_{F}$ it is the former if $\e=1$ and the latter if $\e=-1$).

\item When $m/2n$ is even and $\b=\g^{m/2n}$, then $Q_{\xi}$ is of type (VV). One vertex $i=i^{\dm}$ corresponds to $b_{i}^{2}=\g$, and $\s_{c}|_{L_{i}}=\id$. The factor $\bAut_{L_{i}/k}(M_{i}, \{\cc\}_{i})$ is the Weil restriction of an orthogonal group or symplectic group over $L_{i}$ (when $\s=\id_{F}$ it is the former if $\e=1$ and the latter if $\e=-1$). Another vertex $i'=i'^{\dm}$ corresponds to $b_{i'}^{2}=-\g$, and  $\s_{c}|_{L_{i'}}\ne\id$. The factor $\bAut_{L_{i'}/k}(M_{i'},\{\cc\}_{i'})$ is a unitary group over $k$.

\item When $m/2n$ is even and $\b=-\g^{m/2n}$, then $Q_{\xi}$ is of type (EE). One arrow $e=e^{\dm}:i\to i^{\dm}$ corresponds to $b_{i}^{2}=\g\Nm_{F/k}(\xi^{-1})$, and $\s_{c\xi^{-1}}|_{L_{i}}=\id$. The factor $\frh^{-\e\e_{e}}(M_{i},M_{i^{\dm}})$ is isomorphic to $\wedge^{2}_{L_{i^{\dm}}}(M_{i^{\dm}})$ or $\Sym^{2}_{L_{i^{\dm}}}(M_{i^{\dm}})$ (when $\s=\id_{F}$ it is the former if $\e=1$ and the latter if $\e=-1$). Another arrow $e'=e'^{\dm}:i'\to i'^{\dm}$ corresponds to $b_{i}^{2}=-\g\Nm_{F/k}(\xi^{-1})$, and $\s_{c\xi^{-1}}|_{L_{i}}\ne\id$. The factor $\frh^{-\e\e_{e}}(M_{i'},M_{i'^{\dm}})$ is isomorphic to the space of $L_{i}/k$-Hermitian forms on $M_{i'}$.
\end{enumerate}
\end{itemize}

\subsection{The case $\s|_{k}\ne \id$}. We have $k^{\s}=\CC\lr{\t^{2}}$. Since $c\in F^{\s}$ hence $\g\in k^{\s}$, $\val_{\t}(\g)$ is always even. We have $\g^{m/n}=\b\s(\b)$, which implies that $\val_{\t}(\b)$ is divisible by $m/n$. Hence $L$ splits into $m/n$ factors of $k$. The graph $Q_{\xi}$ is a single cycle of length $m/n$. We distinguish two cases according to the parity of $m/n$.

\sss{$m/n$ is odd} In this case $Q_{\xi}$ is of type (VE).

The unique vertex $i=i^{\dm}$ corresponds to the unique $b_{i}\in k$ such that $b_{i}^{2}=\g$ and $b_{i}^{m/n}=\b$. Since $\s_{c}|_{L_{i}}\ne\id$, the factor $\bAut_{k}(M_{i}, \{\cc\}_{i})$ in $H$ is either an orthogonal group or a symplectic group over $k$.

The unique arrow $e:j\to j^{\dm}$ fixed by $(-)^{\dm}$ corresponds to the unique $b_{j}\in k$  such that $b_{j}^{2}=\g\Nm_{F/k}(\xi^{-1})$ and $b_{j}^{m/n}=\b$.  Since $\s_{c\xi^{-1}}|_{L_{j}}\ne\id$, the factor $\frh^{-\e\e_{e}}(M_{j},M_{j^{\dm}})$ in $\frg(\xi)$ is isomorphic to the space of $k/k^{\s}$-Hermitian forms on $M_{j}$.

\sss{$m/n$ is even} In this case $Q_{\xi}$ is of type (VV) or (EE) according to whether the equations
\begin{equation}\label{eqn b}
b\s(b)=\g, \quad b^{m/n}=\b
\end{equation}
have a common solution in $k^{\times}$.
\begin{itemize}
\item If the equations \eqref{eqn b} have a common solution in $k^{\times}$, then $Q_{\xi}$ is of type (VV). In this case, the two vertices $i,i'$ fixed by $(-)^{\dm}$ correspond to two solutions $b_{i}, b_{i'}=-b_{i}$ to \eqref{eqn b}. The corresponding factors in $H$ are unitary groups over $k^{\s}$.
\item If the equations \eqref{eqn b} do not have a common solution in $k^{\times}$, then $Q_{\xi}$ is of type (EE). In this case, the two arrows $e:i\to i^{\dm},e':i'\to i'^{\dm}$ fixed by $(-)^{\dm}$ correspond to two solutions $b_{i}, b_{i'}=-b_{i}$ to the system of equations
\begin{equation*}
b\s(b)=\g\Nm_{F/k}(\xi^{-1}), \quad b^{m/n}=\b.
\end{equation*}
The corresponding factors in $\frg(\xi)$ are isomorphic to the space of $k/k^{\s}$-Hermitian forms on $M_{i}$ and $M_{i'}$.
\end{itemize}


\begin{thebibliography}{99}
\bibitem[D]{D}V. Drinfeld. Proof of the Petersson conjecture for $\GL(2)$ over a global field of characteristic $p$. {\em Funct. Anal. Appl.} {\bf 22}(1), 28--43(1988).

\bibitem[RLYG]{RLYG}M. Reeder, P. Levy, J.-K. Yu and B. Gross, Gradings of positive rank on simple Lie algebras, {\em Transform. Groups}, {\bf 17} (4), 1123--1190 (2012).


\bibitem[V]{V}E.B. Vinberg, The Weyl group of a graded Lie algebra. {\em Izv. AN SSSR. Ser. Mat.} {\bf 40}(3),
488--526 (1976); English translation: {\em Math. USSR-Izv.} {\bf 10}, 463--495 (1977).

\bibitem[Y]{Y} Z.Yun, Epipelagic representations and rigid local systems.
{\em Selecta Math.} (N.S.), {\bf 22}(3), 1195--1243 (2016).

\end{thebibliography}
\end{document}